\RequirePackage{silence}
\documentclass[11pt, english]{article}

\usepackage[margin= 1.9 cm]{geometry}

\usepackage{amsthm}
\usepackage{amsmath}
\usepackage{amssymb}
\usepackage{setspace}
\usepackage{mathtools}
\usepackage{verbatim}
\usepackage{csquotes}
\usepackage{graphicx}
\usepackage[hidelinks]{hyperref}
\usepackage{thm-restate}
\usepackage{cleveref}
\usepackage{graphicx}
\usepackage{enumitem}
\usepackage{framed}
\usepackage{subcaption}
\usepackage{floatrow}
\usepackage[T1]{fontenc}
\theoremstyle{plain}

\newtheorem*{thm*}{Theorem}
\newtheorem{thm}{Theorem}[section]
\Crefname{thm}{Theorem}{Theorems}

\newtheorem*{lem*}{Lemma}
\newtheorem{lem}[thm]{Lemma}
\Crefname{lem}{Lemma}{Lemmas}

\newtheorem*{claim*}{Claim}

\crefname{claim}{Claim}{Claims}
\Crefname{claim}{Claim}{Claims}

\newtheorem{prop}[thm]{Proposition}
\Crefname{prop}{Proposition}{Propositions}

\newtheorem{cor}[thm]{Corollary}
\crefname{cor}{Corollary}{Corollaries}

\newtheorem{conj}[thm]{Conjecture}
\crefname{conj}{Conjecture}{Conjectures}

\newtheorem{qn}[thm]{Question}
\Crefname{qn}{Question}{Questions}

\newtheorem{obs}[thm]{Observation}
\Crefname{obs}{Observation}{Observations}

\newtheorem{ex}[thm]{Example}
\Crefname{ex}{Example}{Examples}

\theoremstyle{definition}

\Crefname{prob}{Problem}{Problems}

\Crefname{defn}{Definition}{Definitions}

\newtheorem*{defn*}{Definition}

\theoremstyle{remark}

\renewenvironment{proof}[1][]{\begin{trivlist}
\item[\hspace{\labelsep}{\bf\noindent Proof#1.\/}] }{\qed\end{trivlist}}

\newcommand{\ceil}[1]{
    \left\lceil #1 \right\rceil
}
\newcommand{\floor}[1]{
    \left\lfloor #1 \right\rfloor
}

\newcommand{\eps}{\varepsilon}

\newcommand{\pr}{\mathbb{P}}

\expandafter\def\expandafter\normalsize\expandafter{%
    \normalsize
    \setlength\abovedisplayskip{8pt}
    \setlength\belowdisplayskip{8pt}
    \setlength\abovedisplayshortskip{4pt}
    \setlength\belowdisplayshortskip{4pt}
}

\usepackage[square,sort,comma,numbers]{natbib}
 \setlist[itemize]{leftmargin=*}

\newcommand{\G}{\mathcal{G}}
\newcommand{\E}{\mathbb{E}}

\DeclareFontFamily{OT1}{pzc}{}
\DeclareFontShape{OT1}{pzc}{m}{it}{<-> s * [1.10] pzcmi7t}{}
\DeclareMathAlphabet{\mathpzc}{OT1}{pzc}{m}{it}

\title{The power of many colours}

\author{Noga Alon\thanks{Department of Mathematics, Princeton
University, Princeton, USA. Research supported in part by NSF grant
DMS-2154082 and by USA-Israel BSF grant 2018267. Email: \href{mailto:nalon@math.princeton.edu} {\nolinkurl{nalon@math.princeton.edu}}.} \and
Matija Buci\'c\thanks{School of Mathematics, Institute for Advanced Study and Department of Mathematics, Princeton University, Princeton, USA. Research supported in part by NSF grant CCF-1900460 and ERC Grant 101044123 (RandomHypGra). Email: \href{mailto:matija.bucic@ias.edu} {\nolinkurl{matija.bucic@ias.edu}}.} \and 
Micha Christoph\thanks{Department of Computer Science, ETH Z\"urich, Z\"urich, Switzerland. Email: \href{micha.christoph@inf.ethz.ch} {\nolinkurl{micha.christoph@inf.ethz.ch}}.}
\and
Michael Krivelevich\thanks{School of Mathematical Sciences, Tel Aviv University, Tel Aviv, Israel. Research supported in part by USA-Israel BSF grant 2018267. Email:
\href{krivelev@tauex.tau.ac.il} {\nolinkurl{krivelev@tauex.tau.ac.il}}.}
}
 \date{}

\begin{document}
\maketitle

\begin{abstract}

   A classical problem, due to Gerencs\'er and Gy\'arf\'as from 1967, asks how large a monochromatic connected component can we guarantee in any $r$-edge colouring of $K_n$? We consider how big a connected component can we guarantee in any $r$-edge colouring of $K_n$ if we allow ourselves to use up to $s$ colours. This is actually an instance of a more general question of Bollob\'as from about 20 years ago which asks for a $k$-connected subgraph in the same setting. We complete the picture in terms of the approximate behaviour of the answer by determining it up to a logarithmic term, provided $n$ is large enough. We obtain more precise results for certain regimes which solve a problem of Liu, Morris and Prince from 2007, as well as disprove a conjecture they pose in a strong form. 
   

   We also consider a generalisation in a similar direction of a question first considered by Erd\H{o}s and R\'enyi in 1956, who considered given $n$ and $m$, what is the smallest number of $m$-cliques which can cover all edges of $K_n$? This problem is essentially equivalent to the question of what is the minimum number of vertices that are certain to be incident to at least one edge of some colour in any $r$-edge colouring of $K_n$. We consider what happens if we allow ourselves to use up to $s$ colours. We obtain a more complete understanding of the answer to this question for large $n$, in particular determining it up to a constant factor for all $1\le s \le r$, as well as obtaining much more precise results for various ranges including the correct asymptotics for essentially the whole range.
\end{abstract}

\section{Introduction}
What can one say about the monochromatic structures we are guaranteed to find in any $r$-edge colouring of a large graph $G$? This very general extremal question captures a large number of classical extensively studied topics including a large portion of modern Ramsey theory.

There are limitations to how much structure one can ensure by using only a single colour. For example, a very classical result of Gy\'arf\'as \cite{gyarfas1977partition} and F\"uredi \cite{furedi}, essentially answering a question of Gerencs\'er and Gy\'arf\'as from 1967 \cite{gerencser-gyarfas}, tells us that in an $r$-edge colouring of an $n$-vertex complete graph, one can always find a monochromatic tree, or equivalently a connected component in a single colour, of order at least $\frac{n}{r-1},$ and one cannot hope to do better whenever an affine plane of order $r-1$ exists (which is known to exist whenever $r-1$ is a prime power). This question and its generalisations in many different directions have been extensively studied, see, e.g., surveys \cite{gyarfas-survey,other-survey}. Let us, for example, highlight the question of how many monochromatic components one needs to cover every vertex of an $r$-edge coloured $K_n,$ since due to a beautiful duality reduction of F\"uredi \cite{furedi} this question is equivalent to the famous Ryser conjecture concerning cover and matching numbers in $r$-partite hypergraphs \cite{ryser-original}.

A very natural question that arises is how much more power do we get if we allow ourselves to use more than just a single colour? This type of question is very classical, for example, Erd\H{o}s, Hajnal and Rado \cite{EHR} already in 1965 raised a generalisation of the Ramsey question in this direction. Namely, how large a clique using edges of at most $s$ different colours can one guarantee in an $n$-vertex $r$-edge coloured complete graph? This problem and its generalisations have been extensively studied ever since, see e.g.\ \cite{AEGM,path-mc-ramsey,stein-mc-ramsey,conlon2022set,gyarfas-mc-ramsey,chung-mc-ramsey}. 

In this paper, we will explore this ``power of many colours'' paradigm for several classical extremal problems. The first one has as its starting point a folklore observation, usually attributed to Erd\H{o}s and Rado (see e.g.\ \cite{gyarfas-survey}), which states that in any $2$-edge colouring of $K_n$ the edges in one of the colours connect all $n$ vertices. A natural generalisation asks how big a connected monochromatic subgraph can we guarantee if we are using more than two colours. This is precisely the subject of the question of Gerencs\'er and Gy\'arf\'as \cite{gerencser-gyarfas} from 1967 discussed above. While we will mostly focus on what happens with this question if we are allowed to use up to $s$ colours, let us first mention yet another natural generalisation due to Bollob\'as and Gy{\'a}rf{\'a}s \cite{bollobas-high-connected} from 2003, which also leads to the same question from a different angle. They consider how large a $k$-connected monochromatic subgraph can one guarantee in any $r$-edge colouring of $K_n$. See \cite{luczak-high-connected,LMP-monochromatic, lo2020monochromatic}, and references therein, for details on what is known about this problem. In 2003 Bollob\'as (see \cite{LMP-multicolour}) asked the overarching ``power of many colours'' generalisation, namely what happens if one is allowed to use up to $s$ colours when looking for a large highly connected subgraph. The benchmarks for this problem have been established by Liu, Morris and Prince \cite{LMP-multicolour}, for some improvements and extensions see \cite{liu-thesis,liu-person,forbidden-multicolour}.

As already mentioned we will mostly focus on the $k=1$ instance of the question of Bollob\'as. The reasons for this are many-fold. We find it to be perhaps the most natural instance since it is a direct ``power of many colours'' generalisation of the much more classical question of Gerencs\'er and Gy\'arf\'as. Also, when assuming the underlying complete graph is at least moderately large, it seems most of the difficulty of the general problem is already captured by this instance. 

We will denote by {\boldmath $f(n,r,s)$} the answer to the above question, so it is the maximum number of vertices that we can guarantee to connect using edges in up to  $s$ colours in any $r$-edge colouring of $K_n$. 
For example, $f(n,2,1)=n$ is a restatement of the observation of Erd\H{o}s and Rado, and determining $f(n,r,1)$ precisely recovers the monochromatic component problem of Gerencs\'er and Gy{\'a}rf{\'a}s discussed above.

Gy\'arf\'as \cite{gyarfas1977partition} and independently F\"uredi \cite{furedi} showed $f(n,r,1)\ge \frac{n}{r-1}$, that this is tight given certain divisibility conditions on $n$ and $r$ and is always close to tight. See \cite{furedi-survey,gyarfas-survey}, for more details on the plethora of results surrounding this classical problem.
Turning to what is known about $f(n,r,s)$ for larger $s$, one can observe that the usual examples based on affine planes perform very poorly even for $s=2$ since using any two colours one can connect almost the whole graph. 
Here, Liu, Morris and Prince \cite{LMP-multicolour} 
showed $f(n,r,2)\ge\frac{4n}{r+1}$, that this is tight given certain divisibility conditions on $n$ and $r$ and is always close to tight (in some well-defined quantitative sense). See also a paper of Liu and Person \cite{liu-person} obtaining more precise results for the $s=2$ case of the more general $k$-connectivity version of the question. For larger $s$ the best-known lower bound, also due to Liu, Morris and Prince \cite{LMP-multicolour}, was $f(n,r,s) \ge (1-e^{s^2/(3r)})n$, which reduces to $f(n,r,s) \ge \Omega(\frac{s^2}{r})\cdot n$ when $s \le \sqrt{r}$. In terms of upper bounds for very small values of $s$ the best bound they obtain is $f(n,r,s) \le \frac{2^{s}}{r+1} \cdot n,$ given some divisibility conditions, which also implies a slightly weaker bound holds in general. 

Our first result shows this upper bound is essentially tight, closing the gap between the quadratic in $s$ lower bound and exponential in $s$ upper bound for any fixed $s$, and $r$ large enough, determining the answer up to an absolute constant factor.
\begin{thm}\label{thm:f-s-small-intro} 
    Let $s \ge 1,$ $r$ large enough and $n \ge \Omega(r)$ then $f(n,r,s)= \frac{\Theta(2^{s})}{r}\cdot n.$ 
\end{thm}
We actually prove the same result holds for the more general question of Bollob\'as, requiring a $k$-connected subgraph, provided $n$ is large enough compared to $k$ and $r$, see \Cref{thm:f-s-small} for the exact details. This solves a problem posed by Liu, Morris and Prince in 2007 asking whether the behaviour is exponential, polynomial or something in between in $s$ when $r$ is large compared to $s$.
They asked this problem as an easier variant of a conjecture that a more precise version of \Cref{thm:f-s-small-intro} holds for any $s \le \frac12 \log r$. We disprove this conjecture in a strong form, by improving the above-mentioned exponential upper bound in $s$ to an essentially quadratic one. This bound also determines the answer up to a logarithmic factor for any $s \le \sqrt{r}.$

\begin{thm}\label{thm:f-s-medium-intro}
    For $1\le s \le \sqrt r$ and $n$ large enough we have $f(n,r,s)= \widetilde{\Theta}_r\left(\frac{s^2}{r}\right) \cdot n$.
\end{thm}

Let us note that \Cref{thm:f-s-small-intro} holds for $s \le (1-o_r(1))\log \log r,$ and this is essentially tight since  a more precise version of \Cref{thm:f-s-medium-intro} (see \Cref{prop:f-s-medium}) shows it can not hold for $s\ge (1+o_r(1))\log \log r$. 

The best known upper bound, also from \cite{LMP-multicolour}, when $\sqrt{r} \ll s <\frac{r}2$, was $f(n,r,s) \le \ceil{\left(1-\binom{r}{s}^{-1}\right)n}.$ They also show this bound is tight if $r$ is odd and $s=\frac{r-1}{2}$ and it matches the above-stated lower bound from \cite{LMP-multicolour}, up to an absolute constant in the exponent for $s \ge \Omega(r)$. We improve this upper bound for $\sqrt{r \log r} \le s \ll r$ showing the above lower bound is tight up to an absolute constant in the exponent for the whole of this range.

\begin{thm}\label{thm:f-large-s}
    For $\sqrt{r \log r} \le s \ll r$ and $n$ large enough we have $f(n,r,s)= (1-e^{\Theta(s^2/r)})n.$
\end{thm}

We also obtain an improved bound for $\sqrt{r} \ll s \le \sqrt{r \log r}$ although for this regime our new bound is not tight. For a summary theorem collecting the current state of affairs on the question of determining $f(n,r,s),$ including our new results, we point the reader to \Cref{thm:f-main} in \Cref{sec:conc-remarks}.

The starting point for the second question we will explore is a problem dating back at least to the work of  Erd\H{o}s and R\'enyi from 1956 \cite{erdos-renyi}, and it asks, given $n$ and $m$, to determine the smallest number of $m$-cliques which cover all edges of $K_n$. They observed an immediate counting lower bound of $\binom{n}{2}/\binom{m}{2}$ and that it is tight if and only if an appropriate combinatorial design exists, as well as proved a variety of asymptotic results. Let us also mention that this is perhaps the most basic and one of the oldest examples of the extensively studied ``clique covering'' problem which asks for how many cliques one needs to cover all edges of an arbitrary graph (see e.g.\ a survey \cite{clique-covering-survey} for more details on the vast body of work in this direction). The question we will consider in this language asks what happens if instead of insisting that each clique has size $m$ (or equivalently size at most $m$) we restrict the size of the union of any $s$ cliques to be at most $m$. We will, however, focus here on an equivalent ``inverse'' formulation of the problem, which in the original $s=1$ instance was first considered by Mills in 1979 \cite{mills-covering} in order to understand what happens if $n$ is not necessarily very large compared to $m$ and the asymptotic results of Erd\H{o}s and R\'enyi and later ones of Erd\H{o}s and Hanani \cite{erdos-hanani} do not yet kick in. The inverse of the classical problem asks, given a covering of the edges of $K_n$ with $r$ cliques, how large does a largest clique we used need to be? In this formulation, our question asks how large a union of some $s$ cliques needs to be. 

In fact, there is a third natural, equivalent reformulation which uncovers a relation to the largest monochromatic component question of Gerencs\'er and Gy\'arf\'as and the ``power of many colours'' generalisation of it discussed above. The classical $s=1$ instance is equivalent to asking how many vertices are certain to be incident to at least one edge of some colour in any $r$-edge colouring of $K_n$. Indeed, given an $r$-edge colouring of $K_n$ we can obtain a clique covering by taking for each colour a clique consisting of all vertices incident to at least one edge of the said colour. For the other direction, given a clique covering we assign colours to each clique and then colour an edge in the colour of a clique covering it (choosing arbitrarily if there is more than one option). In this setting, there is a natural ``power of many colours'' generalisation which simply asks what happens if instead of insisting on a single colour we allow ourselves to use up to $s$ colours. 

Let us denote by {\boldmath$g(n,r,s)$} the answer to this question, so the minimum number of vertices that are certain to be incident to at least one edge in a colour from some set of $s$ colours in any $r$-edge colouring of $K_n$. We note that determining $g(n,r,1)$ corresponds to the classical problems discussed above and is reasonably well understood. There has been plenty of papers focusing on what happens for small values of $r$ \cite{fort-triple-covering,mills-covering, mills2, horak-sauer-covering, todorov, shiloach-covering}. There are also several results \cite{todorov2,furedi90} obtaining more general, precise results provided $r=p^2+p+1,$ and a finite projective plane of order $p$ exists, for example, if in addition $r \mid n$ then $g(n,r,1)= \frac{p+1}{p^2+p+1} \cdot n.$ F\"uredi \cite{furedi90} proved in 1990 a very nice linear programming reduction to a question involving fractional cover numbers of intersecting hypergraphs which can be used to prove various approximate results on $g(n,r,1)$. In terms of asymptotics, it is fairly easy to see that $g(n,r,1)= \frac{1+o_r(1)}{\sqrt{r}} \cdot n$ for $n$ large enough. The lower bound is the above-mentioned counting argument of Erd\H{o}s and R\'enyi, while the basic idea of using projective planes to show the upper bound dates back to a paper of Erd\H{o}s and Graham \cite{erdos-graham} from 1975, who worked on a closely related Ramsey theoretic question. See also a recent paper spelling this out in more detail \cite{discrepancy-of-spanning-trees}. For more information we point the interested reader to Chapter 7 of a survey of F\"uredi \cite{furedi-survey}.

Regarding $g(n,r,s)$, we determine the answer up to a constant factor for all $1\le s \le r,$ as well as obtain much more precise results for various ranges including the correct asymptotics for essentially the whole range. 

In terms of how the two functions are related, we always have $f(n,r,s)\le g(n,r,s)$ since in a connected component using up to $s$ colours we know that every vertex of the component must be incident to an edge in one of these $s$ colours. In some sense this positions the question of determining $g(n,r,s)$ as the ``$0$-connectivity'' instance of the question of Bollob\'as, where in this context we say a graph is $0$-connected if it does not contain an isolated vertex. With this in mind, similarly, as upper bounds on $f(n,r,s)$ serve as often best-known bounds for any instance of the question of Bollob\'as, upper bounding $g(n,r,s)$ provides us with our improved upper bound on $f(n,r,s),$ which is behind \Cref{thm:f-large-s}. 

The known classical results tell us that $f(n,r,1)$ is roughly equal to $\frac{n}{r}$ while $g(n,r,1)$ is roughly $\frac{n}{\sqrt{r}},$ showing that the functions have very different behaviour for $s=1$. On the other hand, we show the functions have the same asymptotic behaviour for $s \ge \sqrt{r \log r}$. The main additional result for the $g$ function is that we also understand its behaviour for all $s \ll \sqrt{r}$, namely $g(n,r,s)=(1+o_r(1))\frac{s}{\sqrt{r}}\cdot n$ showing it grows linearly in $s$ until the answer becomes an absolute constant times $n$, which at least for very small $s$ is in stark contrast to the $f$ function, which, by \Cref{thm:f-s-small-intro}, initially grows exponentially fast in $s$. The following theorem collects most of our results for the $g$ function and showcases its behaviour across the range. 

\begin{thm}\label{thm:g-main}
Given integers $1\le s \le {r}$ and $n$ large enough compared to $r$ we have
{
    \[
  g(n, r, s) =
  \begin{cases}
    \frac{(1+o_r(1))s}{\sqrt{r}} \cdot n, & \text{for } s \ll \sqrt{r}; \\
    \Theta(n), & \text{for } s= \Theta(\sqrt{r}); \\
    \left(1-e^{-\Theta(s^2/r)}\right)\cdot n, & \text{for } s \ge \sqrt{r \log r}; \\
    \left(1-(1-o_r(1))\cdot\binom{r-s-1}{s}/\binom{r}{s}\right)\cdot n, & \text{for } s \ge (1-o_r(1)) \cdot \frac r2; \\
    \ceil{\left(1-{1}/{\binom{r}{s}}\right)n}, & \text{for } s=\frac{r-1}2 \text{ and } r \text{ odd};\\
    n, & \text{for } s\ge \frac r2.
  \end{cases}
\]
}
\end{thm}

We note that in addition to the colouring formulation we mostly focused on, all our results can be translated to the alternative, clique-cover formulations of the problem mentioned above. There is one final, Ramsey theoretic reformulation which we want to mention.
To this end let us denote by $\mathcal{T}_m$ the set of all trees on $m$ vertices and by $\mathcal{F}_m$ the set of all spanning forests (without isolated vertices) on $m$ vertices. Then determining the $r$-colour Ramsey number of these two families is equivalent to determining $f(n,r,1)$ and $g(n,r,1)$. Let us note that the question of obtaining bounds on Ramsey numbers of an arbitrary member of $\mathcal{F}_m$ as well as $\mathcal{T}_m$ has already been considered in 1975 by Erd\H{o}s and Graham \cite{erdos-graham}, and one can use the same argument to get an approximate understanding of the behaviour of $f(n,r,1)$ and $g(n,r,1)$. If one now considers the ``power of many colours'' generalisation of these two Ramsey problems, namely if instead of a monochromatic copy, we seek one which uses up to $s$ colours, so precisely in the sense introduced by Erd\H{o}s, Hajnal and Rado, and discussed at the beginning of the section, one recovers the problems of determining $f(n,r,s)$ and $g(n,r,s)$.

Most of our arguments use various probabilistic tools and ideas, often both for lower and upper bounds, in certain regimes expansion arguments come in useful. We also prove a number of extremal set theoretic reductions which often lead to interesting questions in their own right, which we collect and discuss further in \Cref{sec:conc-remarks}.

\vspace{-0.1cm}

\subsection{Organisation and notation}\label{sec:organisation}
We begin with our results on the generalised clique covering function $g(n,r,s)$ in \Cref{sec:g}. We open the section by proving two general tools. The first one, \Cref{lem:lower-bound-g}, is a simple probabilistic bound that we will use to prove most of our lower bounds. The second one, \Cref{lem:upper-bounds}, gives a reduction to an extremal set-theoretic problem, which is behind most of our upper bounds. 

In \Cref{sec:g-small} we prove the first two parts of \Cref{thm:g-main}, including a more precise lower bound, given certain divisibility conditions, in the form of \Cref{prop:g-upper-bound}. In the subsequent \Cref{sec:g-large} we prove the remaining parts of \Cref{thm:g-main} concerning larger values of $s$. The third and fourth parts follow from \Cref{prop:g-s-big-lb} and \Cref{cor:g-s-big-ub}. The fifth part follows from \Cref{prop:g-almost-end-of-regime} and \Cref{cor:g-almost-end-of-regime} which is perhaps the most involved of the arguments related to the function $g$ including an intriguing, somewhat non-standard application of Janson's Inequality in tandem with the Lov\'asz Local Lemma. The penultimate part is established in \Cref{g-end-of-regime}. The final part is the easy reduction to the Erd\H{o}s-Rado observation, which was already discussed in the Introduction. In \Cref{sec:small-precise-g} we collect what our results tell us about the behaviour of $g(n,r,s)$ for small values of $r$.

In \Cref{sec:f} we prove our results on the large connected component using few colours function $f(n,r,s)$. We begin with a part of the upper bound of \Cref{thm:f-s-small-intro}, due to \cite{LMP-multicolour}. We include it for completeness and since it helps give motivation for the argument behind our new lower bound, which we prove next as \Cref{thm:f-s-small}. We note that we prove this result in the full generality of the question of Bollob\'as requiring $k$-connectivity of the subgraph we find. We continue with a slightly simpler argument for lower bounding $f(n,r,s),$ which improves the previously best constant in the exponent as \Cref{prop:f-s-medium}, and which combined with our new upper bound for the $g$ function proves \Cref{thm:f-large-s}. We conclude the section with our new upper bound for $f(n,r,s),$ for $s \le \sqrt{r}$ which we prove as \Cref{prop:upr-bnd-f-medium-s}, and which combined with the known lower bound proves \Cref{thm:f-large-s}.

In \Cref{sec:conc-remarks} we make some final remarks and collect a number of interesting open problems.

\textbf{Notation.} Given a graph or a hypergraph $H$, we will denote by $V(H)$ and $E(H)$ its vertex and edge set respectively, and by $d(H)$ its average degree. We use the standard asymptotic notation. Some conventions are that we write $g=O(f)$ if there exists an absolute constant $C>0$ such that $g(x)\le C\cdot f(x)$ for all $x$. We also write  $g=\Omega(f)$ if $f=O(g),$ and $\Theta(f)$ if both $g=O(f)$ and $f=O(g)$. We index the asymptotic notation by a variable we are letting tend to infinity when we wish to emphasise or clarify which one drives the asymptotics. For example, we write $o_r(f)$ to be a function which tends to zero when divided by $f$ as the parameter $r$ tends to infinity. We also write $f \ll g$ to mean $f=o(g)$ when both $f$ and $g$ depend only on a single parameter which is then assumed to tend to infinity. We write $f=\widetilde{\Theta}_r(g)$ if $f= \Omega(g)$ and $f\le O(g) \cdot \log r$. All our logarithms are in base two. Given two events $A$ and $B$ in the same probability space, we write $A \sim B$ if they are dependent. 

\section{Power of many colours for covering vertices}\label{sec:g}
In this section, we prove our results on the generalisation of the clique-covering problem, i.e.\ determining the behaviour of the function $g(n,r,s)$. We note though that most of the upper bounds we will prove in this section are also the best upper bounds we have for the large connected component using few colours function $f(n,r,s)$ using the easy bound $f(n,r,s)\le g(n,r,s)$ explained in the introduction.
We begin with the general tools proven in the following subsection, we then proceed to deduce various parts of our main result, \Cref{thm:g-main}, using them. 

\subsection{General tools}
We begin with the following simple lemma which is behind most of our lower bounds on the generalised clique covering function $g(n,r,s)$.
\begin{lem}\label{lem:lower-bound-g} 
Let $r,s$ and $d$ be positive integers such that $s \le d < r-s$. Then
$$g(n,r,s)\ge \min \left \{1+\frac{s(n-1)}{d}, \left(1-\binom{r-d-1}{s}\Big/\binom{r}{s}\right)n \right \}.$$
\end{lem}
\begin{proof}
Consider an $r$-edge colouring of $K_n$. Our goal is to show that there exist $s$ colours such that there are at least $1+\frac{s(n-1)}{d}$ vertices incident to an edge coloured in one of these $s$ colours or at least $n-n \cdot \binom{r-d-1}{s}/\binom{r}{s}$.   

Suppose first that there is a vertex $v$ for which there are at most $d$ different colours used on the edges incident to $v$. In this case, the $s$ colours with most edges incident to $v$ account for at least $\frac{s(n-1)}{d}$ edges incident to $v$. Counting, in addition, $v$ this implies that there are at least $1+\frac{s(n-1)}{d}$ vertices incident to an edge coloured using one of these $s$ colours, as desired. 

Hence, we may assume that for every vertex $v$ there are at least $d+1$ different colours used on the edges incident to $v$. We will choose our $s$ colours uniformly at random. The probability that we did not pick any of the at least $d+1$ colours incident to $v$ is at most ${\binom{r-d-1}s}/{\binom{r}s}$. 
So the expected number of vertices not incident to an edge using one of our random $s$ colours is at most $n \cdot {\binom{r-d-1}s}/{\binom{r}s}$. In particular, there exists a choice of $s$ colours for which there are at most this many vertices not incident to an edge using one of them. We conclude that there are at least $n-n \cdot \binom{r-d-1}{s}/\binom{r}{s}$ vertices incident to an edge using one of these colours. 
\end{proof}

We note that strictly speaking the assumptions $s \le d$ and $d<r-s$ are not necessary in the above lemma. This follows since the terms in the minimum are decreasing and increasing in $d$ respectively and evaluate to $n$ in the case $d=s$ and $d=r-s$ respectively. We will only ever use the lemma in the specified range since the same reasoning explains one cannot obtain a better bound by choosing $d$ outside this range. 

We next turn to a lemma which will help us prove our upper bounds on $g(n,r,s)$. 

\begin{lem}\label{lem:upper-bounds}
    Let $1\le s \le r$ be integers. Suppose there exists a hypergraph $H$ with vertex set $[r]$ such that \begin{enumerate}
        \item \label{itm:1} for any $e,f \in E(H)$ we have $e \cap f \neq \emptyset$, i.e.\ $H$ is intersecting, and 
        \item \label{itm:2} any set of $r-s$ vertices of $H$ contains at least $t$ distinct edges. 
    \end{enumerate}
    Then $$g(n,r,s) \le n- t \cdot \floor{\frac{n}{|E(H)|}}.$$
\end{lem}

\begin{proof}
Our goal is to colour the edges of $K_n$ into $r$ colours in such a way that for any $s$ colours there are quite a few vertices of $K_n$ not incident to an edge coloured using one of the $s$ colours. To do this we will exploit the existence of the auxiliary hypergraph $H$.  

Let $m:=|E(H)|$ and let us partition the vertices of $K_n$ into $m$ parts $A_e$, one for each edge $e \in E(H)$ with each $A_e$ of size either $\floor{n/m}$ or $\ceil{n/m}$. We note that the $r$ vertices of $H$ represent the $r$ colours we are going to use while an edge $e \in E(H)$ should be viewed as the set of colours we are ``permitted'' to use on the edges of our $K_n$ incident to a vertex belonging to $A_e$. With this in mind, we now colour all the edges of $K_n$ as follows. Given any two distinct vertices belonging to $A_e$ and $A_{e'}$ respectively (with possibly $e=e'$) we colour the edge between them using an arbitrary colour in $e \cap e'$.  We know such a colour always exists since $e \cap e'\neq \emptyset,$ by the assumption that $H$ is intersecting, namely property \ref{itm:1}.

Now to show this colouring satisfies the desired property suppose $S$ is an arbitrary set of $s$ colours. By property \ref{itm:2} we know the set of $r-s$ colours in $V(H) \setminus S$ contains at least $t$ edges, the set of which we denote by $T$. Observe now that, any vertex $v$ belonging to $A_e$ with $e \in T$ is not incident to an edge coloured in a colour in $S$. Indeed, by the definition of our colouring, we are only ever colouring an edge incident to a vertex $v \in A_e$ by a colour which belongs to $e$ and since $e \in T$ we know $e \subseteq V(H) \setminus S,$ so, in particular, no colour in $S$ is allowed to be used. This implies the desired claim since there are at least $t \cdot \floor{n/m}$ vertices of $K_n$ contained in one of the sets $A_e$ with $e \in T$.
\end{proof}

Both above lemmas might seem somewhat wasteful but as we will soon see both can be tight and they suffice to give a pretty good understanding of the behaviour of $g(n,r,s)$ in general. 

Let us also note that if one starts with an $r$-edge colouring of $K_n$ one can define a hypergraph $H$ with the vertex set being the set of $r$ colours and an edge $e(v)$ of $H$ for each vertex $v$ of $K_n$ where $e(v)$ consists of all the colours appearing on edges of $K_n$ incident to $v$. We note that certain edges might appear multiple times so $H$ is strictly speaking a multi-hypergraph. By construction, $H$ must be intersecting and the number of vertices of $K_n$ not incident to an edge of colour in some set $S$ of $s$ colours is precisely equal to $n$ minus the number of edges of $H$ induced by $V(H)\setminus S$. This shows that the question of determining $g(n,r,s)$ is equivalent to finding a multi hypergraph $H$ as in \Cref{lem:upper-bounds} with the maximum possible value of $t$.

Let us also observe that we are not specifying the uniformity of $H$; in fact, we are free to choose it according to our heart's desire. That said, if any edge of $H$ has size at most $s$ then its complement is a set of at least $r-s$ vertices not containing any edges (since $H$ must be intersecting), so in that case the best $t$ we can choose is $0$ and we only obtain the trivial bound. In particular, if we wish to obtain non-trivial bounds we should always choose every edge of $H$ to have a size of at least $s+1$. In most of our applications of the lemma, we will choose $H$ to be $(s+1)$-uniform.

\subsection{Bounds when \texorpdfstring{$s$}{s} is small}\label{sec:g-small}

In this section, we will show our results on $g(n,r,s)$ for $s \le \sqrt{r}$. We begin with an upper bound under some divisibility conditions.
\begin{prop}\label{prop:g-upper-bound}
Let $s$ be a positive integer, and let $p$ be a power of a prime number. Setting $r=p^2+p+1$ we have for any $n$ divisible by $r$
$$g(n,r,s)\le (sp+1)\cdot \frac{n}{r}.$$
\end{prop}

\begin{proof}
Let $r=p^2+p+1$. If $s\ge p+1$ then the inequality is implied by the trivial upper bound of $g(n,r,s)\le n$, so we may assume $s \le p$. By the assumption that $p$ is a power of a prime, we know there exists a projective plane of order $p$. In particular, there exists an $r$-vertex, $r$-edge, $(p+1)$-uniform hypergraph $H$ in which every vertex belongs to exactly $p+1$ edges and every pair of vertices belongs to a unique edge. 

Observe first that if we fix a set $S$ of $s$ vertices of $H$ it intersects at most $sp+1$ edges of $H$. To see this fix some $v \in S$. Note that $v$ is contained in precisely $p+1$ edges, but each $w \in S \setminus \{v\}$ belongs to at most $p$ edges not already accounted for (since one of the edges containing $w$ contains $v$ as well). Therefore, there are at most $p+1+(s-1)p=sp+1$ edges intersecting $S$ in total. 

This in particular implies that any set of $r-s$ vertices contains at least $r-sp-1$ edges of $H,$ so \Cref{lem:upper-bounds} implies 
$$g(n,r,s)\le n-(r-sp-1)\cdot \frac{n}{r}= (sp+1)\cdot \frac{n}{r},$$
as desired.
\end{proof}

\Cref{prop:g-upper-bound} is tight for $s=1$ and $s=2$ by applying \Cref{lem:lower-bound-g} with $d=p$. We will also soon see that it is asymptotically tight for any $s\ll \sqrt{r}$. But, let us first extract from it an approximate upper bound which applies without any divisibility assumptions and is similarly asymptotically optimal for $s\ll \sqrt{r}$ and somewhat large $n$.

\begin{cor}\label{cor:g-s-small-ub}
    Suppose $n \gg r^{3/2}/s$ then
    $$g(n,r,s) \le (1+o_r(1))\cdot \frac{s}{\sqrt{r}}\cdot n,$$
    where the $o_r(1)$ term is a function of $r$ which tends to $0$ as $r \to \infty$.
\end{cor}
\begin{proof}
Note that the desired inequality holds by the trivial upper bound of $g(n,r,s) \le n$ if $s \ge \sqrt{r}$ so we may assume $s<\sqrt{r}.$
Let $r'$ be the largest integer smaller or equal to $r$ of the form $p^2+p+1$ where $p$ is a prime. By well-known results on the difference between consecutive primes (see e.g.\ \cite{primes}) we know that $p=(1-o_r(1))\sqrt{r}.$ Let finally $n'=n- (n \bmod r')$ so that $r' \mid n'$ and $n'=(1-o_r(1))n$.
We now have 
$$g(n,r,s) \le g(n,r',s) \le g(n',r',s)+n-n'\le (sp+1)\cdot \frac{n'}{r'}+r'\le (1+o_r(1))\cdot \frac{s}{\sqrt{r}}\cdot n,$$
where in the penultimate inequality we used \Cref{prop:g-upper-bound} and in the final the assumption that $n \gg r^{3/2}/s$.    
\end{proof}

We note that some lower bound on $n$ in terms of $s$ and $r$ is clearly necessary, 
and one can easily extend the range for which the above result applies by weakening it after replacing the $1+o_r(1)$ term with a $O_r(1)$ term.

We now turn to the lower bounds. We begin by showing that \Cref{cor:g-s-small-ub} is asymptotically tight.
\begin{prop}
Given positive integers $n,r,s$ such that $s \le \sqrt{r},$ we have 
$$g(n,r,s) \ge \frac12 \cdot \frac{s}{\sqrt{r}} \cdot n; $$
furthermore, if $s \ll \sqrt{r}$ then
$$g(n,r,s) \ge \left(1-o_r(1)\right)\cdot \frac{s}{\sqrt{r}} \cdot n.$$
\end{prop}
\begin{proof}
Let us apply \Cref{lem:lower-bound-g} with $d=\floor{\sqrt{r}}\ge s$  this shows that either 
$g(n,r,s) \ge 1+\frac{s(n-1)}{\floor{\sqrt{r}}}\ge \frac{sn}{\floor{\sqrt{r}}}$ in which case we are done, or 
\begin{align*}
    \frac{g(n,r,s)}n&\ge 1-\frac{\binom{r-d-1}{s}}{\binom{r}{s}}\ge 1-\left(1-\frac{d+1}{r}\right)^s \ge \frac{s(d+1)}{r}-\binom{s}{2} \cdot\frac{(d+1)^2}{r^2}\\
    &=\frac{s(d+1)}{r} \cdot \left(1 - \frac{(s-1)(d+1)}{2r}\right)\ge \frac{s}{\sqrt{r}}\cdot \left(1 - \frac{sd}{2r}\right), 
\end{align*}
which in turn is at least $(1-o_r(1)) \cdot \frac{s}{\sqrt{r}}$ for $s \ll \sqrt{r},$ and at least $\frac12 \cdot \frac{s}{\sqrt{r}}$ for $s \le \sqrt{r}$, as desired.
\end{proof}



\subsection{Bounds when \texorpdfstring{$s$}{s} is large}\label{sec:g-large}
In this section, we show our bounds on $g(n,r,s)$ for $s \ge \sqrt{r}$. We begin with the lower bounds.

\begin{prop}\label{prop:g-s-big-lb}
    Let $s \le r$ be positive integers. Then, for any $n$ we have  
$$g(n,r,s) \ge n \cdot \left(1-\binom{r-s-1}{s}\Big/\binom{r}{s}\right)\ge n \cdot \left(1-e^{-\frac{s(s+1)}{r}}\right).$$
\end{prop}

\begin{proof}
    Let us apply \Cref{lem:lower-bound-g} with $d=s$. Since the first term is equal to $n$ and hence always at least as large as the second, we always have
    $$g(n,r,s) \ge  n-n \cdot \binom{r-s-1}{s}\Big/\binom{r}{s} \ge n -n \cdot \left(1-\frac{s+1}{r}\right)^s \ge n \cdot \left(1-e^{-\frac{s(s+1)}r}\right), $$
    as desired.
\end{proof}

The following result will be behind an almost matching upper bound we obtain for $s \ge \sqrt{r \log r}$. 

\begin{prop}\label{prop:g-big-s-ub}
Let $s,r$ be positive integers such that $\sqrt{r \log r} \le s \le \frac{r}{32}$. Then, there exists an $r$-vertex intersecting hypergraph of uniformity $8s$ and with cover number larger than $s$ with at most $e^{O(s^2/r)}$ edges.
\end{prop}
\begin{proof}
    Let $u:=8s$. Let $H$ be a random hypergraph obtained by picking $m=\floor{e^{u^2/(2r)}}$ random edges $e_1,\ldots, e_m \subseteq [r]$, each sampled uniformly at random among all subsets of $[r]$ of size $u$, independently from all other edges.

    First, observe that for any $1\le i,j \le m$ the probability that $e_i$ and $e_j$ do not intersect is at most 
    $$\binom{r-u}{u}/\binom{r}{u} = \frac{(r-u)\cdot \ldots \cdot (r-2u+1)}{r \cdot \ldots \cdot (r-u+1)}< \left(1-\frac{u}{r}\right)^u < e^{-u^2/r}.$$
    So by a union bound the probability that there exists a non-intersecting pair of edges is less than
    $$\binom{m}{2}\cdot e^{-u^2/r}\le \frac{m^2}{2}\cdot e^{-u^2/r} \le \frac12.$$ 

    On the other hand, given a fixed set of $r-s$ vertices, the probability that $e_i$ is contained in it equals 
    $$ \binom{r-s}{u}\Big/\binom{r}{u} \ge \left(1-\frac{s}{r-u+1}\right)^u \ge \left(1-\frac{4s}{3r}\right)^u \ge e^{-{3su}/{r}+1},$$
    where in the penultimate inequality we used $u = 8s \le \frac r4$ and in the final inequality we used that $1-x \ge e^{-2x}$ for any real $0\le x \le \frac12$ and that $\frac{su}{3r}\ge 1$.
    
    Hence, the probability that none of our random $m$ edges is contained in our set of size $r-s$ is at most  
    $$\left(1-e^{-{3su}/{r}+1}\right)^m \le \exp\left(-m\cdot e^{-{3su}/{r}+1} \right)\le \exp\left(-e^{{su}/{r}} \right).$$
    Since there are $\binom{r}{r-s}=\binom{r}{s}\le \frac{r^s}{s!}\le \frac{r^s}2$ subsets of size $r-s,$ another union bound implies that the probability that there exists a set of $r-s$ vertices not containing any edge is at most
    $$\frac{r^s}2\cdot \exp\left(-e^{{su}/{r}} \right)=\frac{1}2\cdot \exp\left(s \log r -e^{{su}/{r}} \right)\le \frac12,$$
    where the final inequality follows since $su/r=8s^2/r \ge 8 \log r \ge \log s + \log \log r$.

    A final union bound implies that there exists an outcome in which the hypergraph we obtain is intersecting and any subset of $r-s$ vertices contains an edge, in other words, there is no cover of size at most $s$, as desired.
\end{proof}

We note that \Cref{prop:g-big-s-ub} is tight up to the constant factor in the exponent as can be seen by an easy double-counting argument. Indeed, there needs to be at least $\binom{r}{r-s}$ pairs of a set of $r-s$ vertices containing an edge of our hypergraph and each edge can belong to at most $\binom{r-u}{r-s-u}$ sets. Therefore, our hypergraph needs to have at least $\binom{r}{r-s}/\binom{r-u}{r-s-u}= \binom{r}{s}/\binom{r-u}{s}\ge \left(1+\frac{u}{r-u}\right)^{s}\ge e^{\Omega(s^2/r)}$ edges.

\begin{cor}\label{cor:g-s-big-ub}
    Let $s,r$ be positive integers such that $\sqrt{r \log r} \le s \le \frac{r}{32}$. Then, there exists $C>0$ such that for any $n\ge e^{Cs^2/r}$ we have 
    $$g(n,r,s) \le \left(1-e^{-O(s^2/r)}\right)\cdot n.$$
\end{cor}
\begin{proof}
\Cref{prop:g-big-s-ub} provides us with an $r$-vertex hypergraph with at most $e^{O(s^2/r)}$ edges which is intersecting and any set of $r-s$ vertices contains at least one edge. This allows us to apply \Cref{lem:upper-bounds} with $t=1$ to obtain
$$g(n,r,s) \le n -\floor{\frac{n}{e^{O(s^2/r)}}}\le \left(1-e^{-O(s^2/r)}\right)\cdot n,$$
as desired.    
\end{proof}

We note that by being a bit more careful in \Cref{prop:g-big-s-ub} one can find a hypergraph of uniformity $u=(2+o_r(1))s$ and ensure via Chernoff's inequality that every set of $r-s$ vertices contains not just one but an $e^{-(1+o_r(1))su/r}$ proportion of the edges of $H$. This allows one to improve the constant in the exponent of the lower order term in the above corollary to $2+o_r(1)$ which is tight up to this factor of about two, thanks to \Cref{prop:g-s-big-lb}. Let us also note that \Cref{cor:g-s-big-ub} is tight in terms of the bound on $n$, up to the constant $C$, since if $n < e^{\frac{s(s+1)}{r}}$ \Cref{prop:g-s-big-lb} implies that $g(n,r,s)=n$.

\textbf{Remark.} We believe \Cref{cor:g-s-big-ub} should hold also when $\sqrt{r}\le s \le \sqrt{r \log r}$. In this case, one can prove a weaker bound than the one in \Cref{cor:g-s-big-ub} which is still better than $g(n,r,s) \le (1-r^{-O(1)})n$ which one gets from monotonicity and \Cref{cor:g-s-big-ub} with $s=\sqrt{r\log r}$. Namely, by choosing $m=r^8$ random edges of size $u=4\sqrt{r \log r}$ we can guarantee the hypergraph is intersecting with probability more than a half and by using Chernoff we can also guarantee with probability at least a half that any set of $r-s$ vertices contains $e^{-O(su/r)}=e^{-O\left(s\sqrt{\frac{\log r}{r}}\right)}$ proportion of the edges. This shows via \Cref{lem:upper-bounds} that in this range $g(n,r,s) \le \left(1-e^{-O\left(s\sqrt{\frac{\log r}{r}}\right)}\right)n$.


The following proposition allows us to obtain an even more precise result when $s$ is close to $r/2$.

\begin{prop}\label{prop:g-almost-end-of-regime}
    Let $s,r$ be positive integers such that $r/2-o_r((r/\log r)^{1/3})\le  s < r/2-1$. Then, there exists an $r$-vertex intersecting hypergraph $H$ of uniformity $s+1$ such that every set of $r-s$ vertices of $H$ contains at least $\displaystyle (1-o_r(1))\frac{\binom{r-s-1}{s}}{\binom{r}{s}}\cdot |E(H)|$ edges.
\end{prop}
\begin{proof}
    Let $\eps>0,$ and let us also for convenience set 
 $$x:=s+1 \hspace{1cm} \text{ and } \hspace{1cm} d:=r/2-x=r/2-s-1=o((r/\log r)^{1/3}).$$

Our general strategy is similar to the one we used in \Cref{prop:g-big-s-ub}. However, in order to obtain a result at our desired level of precision we cannot afford to simply sample every edge with an appropriate probability. What we do is for each $X\in\binom{[r]}{x}$, sample a Bernoulli random variable $z_X$ with probability $p=\frac1{\binom{r-x}{x}} \le \frac12,$ where the inequality holds since $r-x>x$. We now choose $X$ to be an edge if $z_X=1$ but $z_{X'}=0$ for all $X'\in\binom{[r]\backslash X}{x}$. Let $A_X$ be the indicator random variable for whether we selected $X$ as an edge, so 
$$A_X = z_X\cdot\prod_{X'\in\binom{[r]\backslash X}{x}}(1-z_{X'}).$$ 
Note first that $H$ is intersecting. Indeed, if $A_X=1$ for some $X\in\binom{[r]}{x},$ then $z_X=1$ and $z_{X'}=0$ for all $X' \in \binom{[r]\backslash X}{x},$ which in turn guarantees $A_{X'}=0,$ so $H$ contains no edge disjoint from $X$.

Let $B$ be the number of edges of $H$, so  $B=\sum_{X\in\binom{[r]}{x}}A_X.$  
Note that since for any $X\in \binom{[r]}{x}$, $\pr(z_X=1)=p$ and since the variables $z_X$ are independent we have
$$
\mathbb{P}(A_X=1) =p(1-p)^{\binom{r-x}{x}}\ge p \cdot e^{-2p\binom{r-x}{x}}=\frac p{e^2},
$$
where in the inequality we used that $1-p \ge e^{-2p}$ which holds since $p \le \frac12$.
This combined with the linearity of expectation implies
$$
\mathbb{E}[B] = \binom{r}{x} \cdot \mathbb{P}(A_X=1)\ge \binom{r}{x} \cdot \frac p{e^2}. 
$$

Our probability space is given by $\binom{r}{x}$ independent random variables, and changing one of them changes the value of $B$ by at most $\binom{r-x}{x}$. In other words, $B$ is $\binom{r-x}{x}$-Lipschitz, implying by Azuma's inequality (see e.g. \cite[Chapter 7]{alon-spencer}) that 
\begin{align}\label{eq:num-edges}
    \mathbb{P}\left(B>{(1+\eps)\mathbb{E}[B]}\right)&\leq \exp\left(-\frac{\eps^2\mathbb{E}[B]^2}{2\sum_{X\in\binom{[r]}{x}}\binom{r-x}{x}^2}\right) \le \exp\left(-\binom{r}{x}\cdot\frac{\eps^2p^4}{2e^4}\right)\le \exp\left(-\binom{r}{x}\cdot\Omega(r^{-8d})\right), 
\end{align}
where we used that $\frac1p = \binom{r-x}{x} = \binom{r-x}{r-2x}=\binom{r-x}{2d} \le r^{2d}$.

We next turn to the number of edges contained by vertex subsets of size $r-s$. 
For each $T\in\binom{[r]}{r-s}$ let $B_T$ count the number of edges of the subgraph of $H$ induced by $T$, so $B_T = \sum_{X\in\binom{T}{x}}A_X$.
Hence, for any $T$
\begin{align}\label{eq:expectation}
\mathbb{E}[B_T] = \binom{|T|}{x} \cdot \mathbb{P}(A_X=1)\ge \binom{r-s}{x} \cdot \frac p{e^2}= \frac {\binom{r-s}{x}}{e^2 \cdot \binom{r-s-1}{x}}=\frac 1{e^2} \cdot \frac{r-s}{r-s-x}> \frac 1{2e^2} \cdot \frac{r}{2d+1}.
\end{align}

Now, let us fix some $T\in \binom{[r]}{r-s}$. Note that the random variable $B_T$ is increasing in the product probability space given by random variables $z_X$ for $X\in\binom{T}{x}$ and $1-z_X$ for $X\in\binom{[r]}{x}\backslash\binom{T}{x}$  since $r-s<2x$ and therefore, for no $X\in T$ does $A_X$ depend on $z_{X'}$ for some $X'\in \binom{T}{x}\setminus\{X\}$. This will allow us to apply Janson's inequality (see \cite[Chapter 8]{alon-spencer}). Before applying it we first need to understand the dependency structure.

Given some $X,X'\in\binom{T}{x}$, $A_X$ and $A_{X'}$ are functions of disjoint sets of independent variables if and only if $r-|X\cup X'|< x$ since this is equivalent to there being no $X''\in\binom{[r]}{x}$ disjoint from both $X$ and $X'$ (and $r-s<2x$ guarantees $X$ and $X'$ are not disjoint). It follows that $A_X$ is not independent of $A_{X'}$ only if 
$$r-x\geq|X\cup X'|=|X|+|X'|-|X\cap X'| = 2x-|X\cap X'|.$$
This is equivalent to $|X\cap X'|\geq 3x-r=x-2d$. Therefore, for each $X\in \binom{T}{x}$, $A_X$ is independent of all but at most 
\begin{align*}
\sum_{i=x-2d}^{x}\binom{x}{i}\binom{r-s-x}{x-i}&=\sum_{i=x-2d}^{x}\binom{x}{x-i}\binom{2d+1}{x-i}\\ 
&= \sum_{i=0}^{2d}\binom{x}{i}\binom{2d+1}{i}\\
&\le (2d+1)\cdot \sum_{i=0}^{2d}\binom{x}{x-i}\binom{2d}{i}\\
&=(2d+1)\cdot  \binom{x+2d}{x}= (2d+1)\cdot  \binom{r-x}{x} =\frac{2d+1}{p},
\end{align*}
 events $A_{X'}$. The above allows us to calculate the following term, with the goal of applying the lower tail version of Janson's inequality \cite{janson}:
$$
\Delta(B_T) = \sum_{\substack{X,X'\in \binom{T}{x}\\ \text{s.t.\ $A_X\sim A_{X'}$}}}\mathbb{E}[A_XA_{X'}] \leq \mathbb{E}[B_T] + \binom{r-s}{x}\cdot \frac{2d+1}{p} \cdot \pr[A_X]\cdot p=(2d+2)\cdot \mathbb{E}[B_T],
$$
where in the inequality we used that $\pr(A_X \cap A_{X'}) \le \pr[A_X] \cdot p,$ which holds since for $A_{X'}$ to happen we need to have $z_{X'}=1,$ which happens with probability $p$, independently of $A_X$ since $r-s-x<x$.
We get 
\begin{equation*}
\frac{\mathbb{E}[B_T]^2}{\Delta(B_T)}\geq \frac{\mathbb{E}[B_T]}{2d+2} \ge \frac{r}{2e^2 (2d+1) (2d+2)},
\end{equation*}
where we used \eqref{eq:expectation} in the final inequality.
Finally, the lower tail version of Janson's inequality \cite{janson} tells us
\begin{equation}\label{eq:exp-b-t}
\mathbb{P}\left [B_T\leq {(1-\eps)\mathbb{E}[B_T]}\right ]\leq \exp\left(-\frac{\eps^2\mathbb{E}[B_T]^2}{2\Delta(B_T)}\right) \le \exp\left(-\Omega\left(\frac r{d^2}\right)\right).
\end{equation}

The last step is to apply the quantitative version of the Lov\'asz Local Lemma (see \cite[Chapter 5]{alon-spencer}) with ``bad events'' being $B_T< (1-\eps){\mathbb{E}[B_T]}$ for $T \in \binom{[r]}{r-s}$ to show that all these events can be avoided with probability which is greater than the probability that $B>(1+\eps)\mathbb{E}[B]$ estimated in \eqref{eq:num-edges}.
Towards this, we first need to estimate, given $T\in \binom{[r]}{r-s}$,  
for how many $T'\in \binom{[r]}{r-s}$ do $B_T$ and $B_{T'}$ both depend on the same $z_X$ for some $X\in \binom{[r]}{x}$, since $B_T$ is independent of all other $B_{T'}$. Towards this, note that $B_T=\sum_{X \in \binom{T}{x}} A_X,$ and that each $A_X$ only depends on $z_{X}$ and $z_{X'}$ for $X'$ disjoint from $X$. This means that $B_T$ depends only on $z_X$ when $X \subseteq T$ or $|T \setminus X|\ge x$ (as otherwise $X$ is not disjoint from any subset of $T$ of size $x$). To summarise, $B_T$ (and similarly $B_{T'}$) depends on $z_X$ if and only if one of the following holds:
\begin{itemize}
    \item $X\subseteq T$ or
    \item $|T\backslash X|\geq x,$ or equivalently $|T\cap X|\leq r-s-x=2d+1$.
\end{itemize}
It follows that there exists a $z_X$ such that both $B_T$ and $B_{T'}$ depend on it only if one of the following holds:
\begin{enumerate}
    \item \label{itmm:1} $|T\cap T'| \geq x$ \hspace{1.6cm} (there exists $X\in \binom{[r]}{x}$ such that $X\subseteq T \cap T'$);
    \item \label{itmm:2} $|T\cap T'| \leq  4d+2$ \hspace{0.83cm} (there exists $X$ with $|X\cap T| \leq 2d+1$ and $X\subseteq  T'$ or vice versa);
    \item \label{itmm:3} $|T\cap T'| \geq r/2-3d$ \hspace{0.48cm} (there exists $X$ with $|X\cap T|,|X\cap T'| \leq 2d+1$);
\end{enumerate}
where \cref{itmm:2} follows since $|X \cap T| \le 2d+1$ and $X \subseteq T'$ imply $x-(2d+1) \le |X \cap (T'\setminus T)| \le r-s-|T \cap T'|$. 
Similarly \cref{itmm:3} follows since $|X\cap T|,|X\cap T'| \leq 2d+1$ imply that 
$x \le r-|T \cup T'|+|X \cap T| + |X \cap T'|$ which is equivalent to $r/2-3d \le |T \cap T'|$.

If we fix $T$ of size $r-s,$ then there are 
$$\binom{r-s}{i} \binom{s}{r-s-i}=\binom{r-s}{i} \binom{s}{2s-r+i}=\binom{r-s}{i} \binom{s}{i-2d-2}$$ 
choices for $T'$ of size $r-s$ such that $|T \cap T'|=i$. 

Since $x=r/2-d$ \cref{itmm:1} is implied by \cref{itmm:3} so $B_T$ is independent of all but at most 
\begin{align*}
\sum_{i=2d+2}^{4d+2} \binom{r-s}{i} \binom{s}{i-2d-2}+\sum_{i=r/2-3d}^{r-s} \binom{r-s}{r-s-i} \binom{s}{r-s-i} &\le \\ (2d+1)\binom{r-s}{4d+2} \binom{s}{2d}+\sum_{i=0}^{4d+2} \binom{r-s}{i} \binom{s}{i}
&\le \\ 2(4d+2) \binom{r-s}{4d+2} \binom{s}{4d+2}&\le r^{8d+4}
\end{align*}
 $B_T'$. Note next that since $d^3 = o(r/\log r)$, we get from \eqref{eq:exp-b-t}
$$
\mathbb{P}[B_T\leq (1-\eps)\mathbb{E}[B_T]]\leq \exp\left(-\Omega\left(\frac r{d^2}\right)\right) = r^{-\omega(d)}\le r^{-8d-4}(1-r^{-8d-4})^{r^{8d+4}},
$$
allowing us to apply the Lov\'asz Local Lemma to conclude
\begin{align*}
\mathbb{P}\Bigg[\bigcap_{T\in\binom{[r]}{r-s}} B_T> (1-\eps)\mathbb{E}[B_T]\Bigg] &\ge (1-r^{-8d-4})^{\binom{r}{r-s}}\ge  \exp\left(-2\binom{r}{r-s}/r^{8d+4}\right)>  \exp\left(-\binom{r}{x}\cdot\Omega(r^{-8d})\right)\\&\ge \mathbb{P}\left(B>{(1+\eps)\mathbb{E}[B]}\right).
\end{align*}
This implies that there exists an outcome in which for each $T \in \binom{[r]}{r-s}$ we have $B_T > (1-\eps) \E[B_T],$ and, $B \le (1+\eps)\E[B]$. In other words, there exists an $(s+1)$-uniform intersecting hypergraph with $r$ vertices such that any $r-s$ vertices contain at least $(1-\eps) \E[B_T]$ edges, while the whole hypergraph has at most $(1+\eps)\E[B]$ edges. 
The conclusion follows from $$\frac{(1-\eps)\E[B_T]}{(1+\eps)\E[B]}\geq (1-\eps)^2 {\frac{\binom{r-s}{x}}{\binom{r}{x}}}=(1-\eps)^2\frac{\binom{r-s-1}{s}}{\binom{r}{s}},$$ and letting $\eps \to 0$.
\end{proof}
\begin{cor}\label{cor:g-almost-end-of-regime}
    Let $s,r$ be positive integers such that $r/2-o_r((r/\log r)^{1/3})\le  s < r/2-1$. Then for large enough $n$, $$g(n,r,s)\leq n\left(1-(1-o_r(1))\cdot\binom{r-s-1}{s}\Big/\binom{r}{s}\right).$$
\end{cor}
\begin{proof}
\Cref{prop:g-almost-end-of-regime} provides us with an $r$-vertex intersecting hypergraph $H$ such that any set of $r-s$ vertices contains at least $\displaystyle (1-o_r(1))\frac{\binom{r-s-1}{s}}{\binom{r}{s}}\cdot |E(H)|$ edges. \Cref{lem:upper-bounds} applied with $\displaystyle t=(1-o_r(1))\frac{\binom{r-s-1}{s}}{\binom{r}{s}}$ now implies 
$$g(n,r,s)\le n- (1-o_r(1)) \frac{\binom{r-s-1}{s}}{\binom{r}{s}} \cdot |E(H)| \cdot \floor{\frac{n}{|E(H)|}}\le n- (1-o_r(1)) n  \cdot \frac{\binom{r-s-1}{s}}{\binom{r}{s}},$$
where we used the assumption that $n$ is sufficiently large compared to $r$ (and hence also $|E(H)|$). 
\end{proof}
Finally, for the very end of the range we obtain the following simple, precise result.
\begin{prop}\label{g-end-of-regime}
For any odd $ r \ge 3$ we have
$$g(n,r,\floor{r/2})= \ceil{\left(1-\frac{1}{\binom{r}{\floor{r/2}}}\right)n}.$$
\end{prop}
\begin{proof}
Let $s=\floor{r/2}$. For the upper bound, we use \Cref{lem:upper-bounds} with the auxiliary hypergraph being the complete $(s+1)$-uniform hypergraph on $2s+1$ vertices.

For the lower bound, we use \Cref{lem:lower-bound-g} with $d=s,$ which shows $$g(n,2s+1,s)\ge {\left(1-\frac{1}{\binom{2s+1}{s}}\right)n},$$ and since $g(n,2s+1,s)$ is always an integer we may put a ceiling function here.
\end{proof}

\subsection{Precise results for small values of \texorpdfstring{$r$}{r}}\label{sec:small-precise-g}
Given the interest various reformulations of the question of determining $g(n,r,1)$ attracted for the small values of $r$ we collect here what our various results imply for several small values of $r$. We note that for $r \le 4$ the behaviour is completely determined by the known results for $g(n,r,1)$ and the observation that $g(n,r,s)=n$ whenever $s \ge \frac r2$. Hence, in the following theorem we solve the next three cases, namely for $5\le r \le 7$ given certain divisibility conditions on $n$. For $s=1$ it is known and not too hard to remove the divisibility conditions although including it in the following statement would make the statement somewhat unwieldy. We suspect it would not be too hard to figure out the precise answers regardless of the divisibility conditions on $n$ in the $s=2$ cases below, although we do not do this here. Let us also note that the below results always serve as lower bounds regardless of the divisibility conditions.

\begin{thm} We have

\resizebox{\textwidth}{!}{
\begin{tabular}{ll|ll|ll}
$g(n,5,1)=\frac59n$                &\quad if $9 \mid n$,\quad\quad  &\quad $g(n,6,1)=\frac12n$               &\quad if $4 \mid n$, \quad\quad  &\quad $g(n,7,1)=\frac37n$               &if $7 \mid n$,\\[0.2cm]
$g(n,5,2)=\ceil{\frac9{10}n}$,     &\quad
&\quad $g(n,6,2)= \frac45n$              &\quad if  $10\mid n$,
&\quad$g(n,7,2)=\frac57n$               &if  $7 \mid n$,\\[0.2cm] 
$g(n,5,s)=n$                       &\quad $\forall s \ge 3$,
&\quad $g(n,6,s)=n$                      &\quad $\forall s \ge 3$,
&\quad$g(n,7,3)=\ceil{\frac{34}{35}n}$  &\quad \\[0.2cm]
&&&&\quad$g(n,7,s)=n$                  &$\forall s > 3$.    
\end{tabular}
}
\end{thm}  

\begin{proof}
    Curiously $g(n,5,1)$ is the only result on the above list for which the lower bound does not follow from \Cref{lem:lower-bound-g} directly, so we point the reader to e.g.\ \cite{shiloach-covering}. For the upper bound one may take the (multi) hypergraph with vertex set $[5]$ and edge (multi) set $\{145,145,145,234,234,235,235,12,13\}$, this is an intersecting hypergraph in which any four vertices contain at least $4$ edges, implying via \Cref{lem:upper-bounds} that $g(n,5,1) \le n- 4 \cdot \floor{\frac n9}= \frac 59 n$ if $9 \mid n$.

    The lower bounds for $g(n,6,1), g(n,6,2), g(n,7,1)$ and $g(n,7,2)$ all follow from \Cref{lem:lower-bound-g} by choosing $d=2$. The upper bound on $g(n,6,1)$ follows by taking the Fano plane and removing a vertex to obtain a six-vertex hypergraph with four edges in which every vertex has degree two so any five vertices contain at least two edges. The upper bound on $g(n,6,2)$ follows by applying \Cref{lem:upper-bounds} to the hypergraph with vertex set $[6]$ and edge set $\{123,124,346,345,256,135,245,236,146,156\}$ which has the property that any four vertices contain at least two edges. The upper bounds on $g(n,7,1)$ and $g(n,7,2)$ follow from \Cref{prop:g-upper-bound}. The results for $g(n,5,2)$ and $g(n,7,3)$ follow from \Cref{g-end-of-regime}.
\end{proof}

\section{Power of many colours for connected components}\label{sec:f}
In this section, we turn to our results on the large connected component using only a few colours question, namely determining $f(n,r,s)$. 

\subsection{Behaviour for very small \texorpdfstring{$s$}{s}}
We start with the following simple upper bound from  \cite{LMP-multicolour}. We include its short proof here for completeness and since it motivates the argument behind the subsequent lower bound.

\begin{prop}[\cite{LMP-multicolour}]\label{prop:f-uppr-bound-small-s}
    Suppose $r =2^d-1$ and $1\le s \le d$ then we have 
    $f(n,r,s) \le 2^s \cdot \ceil{\frac{n}{r+1}}.$
\end{prop}
\begin{proof}
    Let $V=\mathbb{Z}_2^d$, so $|V|=r+1$. Let us consider the complete graph with vertex set $V$ and let us colour the edge between distinct $\textbf{x},\textbf{y} \in \mathbb{Z}_2^d$ in colour identified with $\textbf{x}-\textbf{y}$. This gives a colouring of an $(r+1)$-vertex complete graph using $r$ colours. 
    
    The key property of this colouring is that in any $s$ colours the graph consisting only of edges using these $s$ colours is a vertex disjoint union of connected components, each consisting of at most $2^s$ vertices. Indeed, any set of $s$ colours corresponds to a subset $S \subseteq \mathbb{Z}_2^d$ of size $s$ and the components correspond to cosets of $\text{Span}(S)$ which partition the vertices and have size at most $2^s$ as claimed.
    
    Now let us consider a complete graph $K_n$ and partition its vertices into $r$ parts of sizes as equal as possible, so each of size $\floor{\frac{n}{r+1}}$ or $\ceil{\frac{n}{r+1}}.$ We then identify the parts with $V$ and colour the edges between parts according to the colouring described above. We colour the edges inside the parts in an arbitrary colour used between it and some other part. In this colouring for any $s$ colours at most $2^s$ parts become connected showing that the maximum size of a connected component in the graph consisting only of edges using these $s$ colours is at most $2^s \cdot \ceil{\frac{n}{r+1}}$, as claimed.
\end{proof}

We next show that the upper bound in the above proposition is tight up to a constant factor for very small values of $s$. We will prove it in the full generality of the question of Bollob\'as, so instead of finding just a large connected component using few colours, we find a highly connected subgraph using edges of only a few colours. To help us find such highly connected subgraphs we are going to use the following classical theorem of Mader, see \cite[Theorem 1.4.3]{diestel}.

\begin{thm}[Mader 1972]\label{mader}
    For every integer $k\ge 1,$ any graph $G$ with $d(G)\ge 4k$ has a $(k+1)$-connected subgraph $H$ with $d(H)>d(G)-2k$.
\end{thm}

The main idea behind our lower bound result for small $s$ is to show that instead of having just a single large highly connected subgraph using up to $s$ colours we will find many such subgraphs, using the same set of $s$ colours. While this is in a sense stronger than what we need, the main benefit is for the next step, in which we want to again find many even larger subgraphs using some fixed $s+1$ colours. 
We begin with a lemma which quantifies how many such large highly connected subgraphs one can find in a single colour. 
We note that a result in a similar direction was derived in \cite{michael-connectivity} and turned out to be quite useful over the years, see e.g.\ \cite{michael-2}. The key distinction in the following lemma is that it works with average degree assumption rather than a minimum degree one and that it tracks the tradeoff between the actual size of the components and the number of them we can find. 

\begin{lem}\label{lem:many-same-col-components}
    Let $r \ge 2, k \ge 1, n > 16r^2(k-1)+1$ and $G$ be an $n$-vertex graph with at least $\frac1r \binom{n}{2}$ edges. Then, there exists $j \ge 2$ such that we can find at least $\frac{r}{4^{j} \log r}$ vertex disjoint $k$-connected subgraphs, each with at least $2^{j-3}\cdot \frac{n}{r} \ge 2$ vertices.
\end{lem}
\begin{proof}
    We will construct our subgraphs greedily. Let $V_1$ be the vertex set of a largest $k$-connected subgraph of $G$. In general, given $V_1,\ldots, V_{i-1}$ we let $V_i$ be the vertex set of the largest $k$-connected subgraph of $G \setminus (V_1 \cup \ldots \cup V_{i-1})$, so long as such a subgraph exists of size more than $\frac{n}{2r}+1$. Let $V_1 \cup \ldots \cup V_t$ be the disjoint sets that the above process produced, and let $V'=V(G) \setminus (V_1 \cup \ldots \cup V_{t})$. Let us also write $s_i=|V_i|$.

    Since by adding to a $k$-connected graph a vertex adjacent to at least $k$ of its vertices keeps the graph $k$-connected, we know that the total number of edges in $G \setminus (V_1 \cup \ldots \cup V_{i-1})$ with at least one endpoint in $V_i$ is at most $\binom{s_i}{2}+n(k-1)$. 
    Turning to $V',$ if the average degree in $G[V']$ is at least $\frac{n}{2r}+2k-2$ then we claim it contains a $k$-connected subgraph with at least $\frac{n}{2r}+1$ vertices, contradicting the definition of $V'$. If $k=1$ this holds by taking a vertex of maximum degree and all of its at least $\frac{n}{2r}$ neighbours. For $k \ge 2,$ by Mader's Theorem (\Cref{mader}) it contains a subgraph with connectivity at least $k$, since $\frac{n}{8r} \ge k,$ and average degree larger than $\frac{n}{2r}$. Hence, we may conclude that the average degree of $G[V']$ is less than $\frac{n}{2r}+2k-2\le \frac{n-1}{r}$ (which guarantees $V' \neq V(G)$ and $|V'|\le n-1$) and in particular $G[V']$ has at most $\frac{n(n-1)}{4r}+(k-1)n$ edges.

    Using this estimate on the number of edges we get that there are at most 
    $$\frac 1r \binom{n}{2} \le \frac{n(n-1)}{4r}+(k-1)n+\sum_{i=1}^t \left(\binom{s_i}{2}+n(k-1) \right) \le \frac{1}{2r}\binom{n}{2}+(t+1)(k-1)n+\sum_{i=1}^t \binom{s_i}{2}$$
    edges in $G$. Since $s_i > \frac{n}{2r},$ we have $t < 2r$. So $(t+1)(k-1)n \le 2r(k-1)n \le \frac{1}{4r}\binom{n}{2},$ using that $n>16r^2(k-1)$. This implies 
    $$\sum_{i=1}^t \binom{s_{i}}{2} \ge \frac{1}{4r} \cdot \binom{n}{2},$$ which in turn implies $s_1^2+\ldots+s_t^2 \ge \frac{n^2}{4r}$.

    Since for all $ i \in [t]$ we have $n \ge s_i \ge \max\left\{\frac{n}{2r},2\right \}$ we know that there exists $2\le j \le \ceil{\log r}+2$ such that the set $I \subseteq [t]$ consisting of all $i$ such that $s_i \in \left[2^{j-1}\cdot \frac{n}{4r},2^{j}\cdot \frac{n}{4r}\right]$ satisfies  
    $$\frac{n^2}{4r (\log r +2)} \le \sum_{i \in I} s_i^2 \le |I|\cdot 2^{2j}\cdot \frac{n^2}{16r^2}.$$
    This implies $|I| \ge \frac{r}{ 4^{j}\log r}$. Finally, note that since $s_i \ge 2$ we may also assume that $2^{j-1} \cdot \frac{n}{4r} \ge 2$.
\end{proof}

Let us note that if $k=1$ the above lemma does not require $n$ to be large, besides the trivial bound $n>0$. This will be important in our applications since we are going to apply the above lemma multiple times and may lose control of the number of vertices we are still working with after the first application; luckily it will suffice for us to apply the lemma with $k=1$ from the second instance on. 

Before turning to the main result of this section we need a final prerequisite, an immediate observation determining the bipartite Ramsey number for matchings. 

\begin{obs}\label{bip-ramsey-matching}
    Given positive integers $r,k$ and $n > r(k-1)$ in any $r$-colouring of $K_{n,n}$ we can find a monochromatic matching of size $k$.
\end{obs}
\begin{proof}
    Pick an arbitrary perfect matching and observe it contains more than $r(k-1)$ edges. By the pigeonhole principle, this means some $k$ of these edges have the same colour, providing us with the desired matching.
\end{proof}
It is easy to see that the above lemma is tight, we simply split one side into $r$ parts of size at most $k-1$ each and colour all edges incident to part $i$ in the colour $i$. 

We are now ready to prove our lower bound theorem for small values of $s$. We will denote by $f(n,r,s,k)$ the largest number of vertices in a $k$-connected subgraph using up to $s$ colours that we can guarantee in any $r$-edge colouring of $K_n$.

\begin{thm}\label{thm:f-s-small}
    Suppose $1\le s \le (1-o_r(1))\log \log r$ and $n>16r^2(k-1)+1,$ then we have 
    $$f(n,r,s,k) \ge 2^{s-2} \cdot \frac{n}{r}.$$
\end{thm}
\begin{proof}
We will prove that the desired inequality holds for $s \le \log \log r-\log (4+3\log \log r).$

Let $n_0=n$ and let us apply \Cref{lem:many-same-col-components} to the graph consisting of the edges in the majority colour $c_1$ to find $j_1\ge 2$ such that there exists at least $n_1 = \ceil{\frac{r}{4^{j_1} \log r}}$ $k$-connected subgraphs, each with at least $s_1\ge 2^{j_1-3}\cdot \frac nr \ge 2$ vertices. We now contract each of these graphs into a single vertex to obtain a (multi) coloured $K_{n_1}$ where each vertex corresponds to a set of at least $s_1$ vertices in our original graph, spanning a $k$-connected subgraph in colour $c_1$. Next, we keep an edge of only one colour $c$ between any two vertices $u,v\in V(K_{n_1}),$ where $c$ is selected in the following way. Let $U,V\subseteq V(K_n)$ be the sets of vertices in our original graph corresponding to $u$ and $v$ respectively. We choose $c$ to be such that the original graph contains a matching of size $k$ in $c$ between $U$ and $V$. Note that there exists such a colour by \Cref{bip-ramsey-matching} since $s_1 \ge \frac{n}{2r} > r(k-1)$. We then repeat, applying \Cref{lem:many-same-col-components}, but from now on always with $k=1$, to the subgraph consisting of the majority colour in the current $K_{n_i}$, for as long as we have at least two vertices left. So, after step $i$ for some $j_i\ge 0$ we have at least $n_i = \ceil{\frac{r}{4^{j_i} \log r}}$ vertices each corresponding to at least $s_1 \cdot s_2 \cdots s_i$ vertices of our original $K_n,$ which induce a $k$-connected graph using the edges of colours $c_1,\ldots,c_i$ and where $s_i\ge 2^{j_i-3}\cdot \frac{n_{i-1}}{r}\ge 2$. Note that placing a matching of size $k$ between two $k$-connected graphs produces a new $k$-connected graph. Therefore, our choice for the colouring of $K_{n_1}$ guarantees that a connected component in any colour $c$ in $K_{n_i}$, when $i \ge 2$, corresponds to a $k$-connected subgraph of our original graph using colours $c_1,\ldots,c_{i-1}$ and $c$.

Let us suppose that the process runs for $t$ steps so that $n_t=1$ and $n_i \ge 2$ for all $i < t$. Since $n_t=1$ we have $\frac{r}{\log r} \le 4^{j_t},$ or after taking a logarithm that $j_t \ge \frac 12 \log r -\frac12 \log \log r$. Using that $n_{i-1} \ge 2$ for all $i \le t,$ we get 
$$2 \le 2^{j_i-3}\cdot \frac{n_{i-1}}{r}\le 2^{j_i-3}\cdot \frac{\ceil{\frac{r}{4^{j_{i-1}} \log r}}}{r}\le 2^{j_i-3}\cdot \frac{2}{4^{j_{i-1}} \log r}.$$ 
After taking a logarithm of both sides this implies $j_i \ge 2j_{i-1}+3+\log \log r$.

If $t \ge s,$ then we know that after $s$ steps we connected at least $s_1 \cdot s_2 \cdots s_s \ge 2^{s-2} \cdot \frac nr$ vertices, since $s_1 \ge \frac 12 \cdot \frac nr$ and $s_i \ge 2$ for all $i$, and are done. So we may assume that the total number of steps $t$ satisfies $t<s,$ and hence it suffices to show that $s_1 \cdot s_2 \cdots s_t \ge 2^{s-2} \cdot \frac nr.$  

Finally, we prove by induction on $i$ that for $i \le t$ we have $s_1 \cdot s_2 \cdots s_i \ge 2^{j_i/2^{i-1}-(1-1/2^{i-1})(3+\log \log r)} \cdot \frac n{8r}.$ Indeed the base case $i=1$ holds since $s_1 \ge 2^{j_1-3} \cdot \frac nr$ and for $i>1$ we get
\begin{align*}
    s_1 \cdot s_2 \cdots s_i &\ge 2^{j_{i-1}/2^{i-2}-(1-1/2^{i-2})(3+\log \log r)} \cdot \frac n{8r} \cdot s_i \\
    &\ge 2^{j_{i-1}/2^{i-2}-(1-1/2^{i-2})(3+\log \log r)} \cdot \frac n{8r} \cdot  \frac{2^{j_i-3}}{4^{j_{i-1}} \log r}\\
    &=2^{j_i-(2-1/2^{i-2})(j_{i-1}+3+\log \log r)} \cdot \frac n{8r}\\
    &\ge 2^{j_i/2^{i-1}-(1-1/2^{i-1})(3+\log \log r)} \cdot \frac n{8r},
\end{align*}
where in the second inequality we used $s_i \ge 2^{j_i-3}\cdot \frac{n_{i-1}}{r} = 2^{j_i-3}\cdot \frac{\ceil{\frac{r}{4^{j_{i-1}} \log r}}}{r}\ge \frac{2^{j_i-3}}{4^{j_{i-1}} \log r},$ and in the final inequality we used $(1-1/2^{i-1})j_i \ge (1-1/2^{i-1})(2j_{i-1}+3+\log \log r).$ This completes the induction step.

Applying this inequality for $i=t$ gives $s_1 \cdot s_2 \cdots s_t \ge 2^{j_t/2^{t-1}-3-\log \log r} \cdot \frac n{8r}>2^{s-2} \cdot \frac nr$, since $$\frac{j_t}{2^{t-1}}-3-\log \log r \ge \frac{\log r  - \log \log r}{2^{t}}-3-\log \log r \ge \frac{\log r}{2^{t}}-3-2\log \log r > \log \log r +1 \ge s+1,$$
where in the penultimate inequality we used $t<s \le \log \log r-\log (4+3\log \log r)$.
\end{proof}

\subsection{Behaviour for larger \texorpdfstring{$s$}{s}}
The following proposition is a slight improvement compared to a similar result from \cite{LMP-multicolour} showing that $f(n,r,s) \ge \left(1-e^{-\frac{s^2}{3r}}\right)\cdot n.$ We include it here mostly for completeness and since it is a very short and nice argument and as we will see is close to being tight. 

\begin{prop}\label{prop:f-s-medium}
    For any $1\le s \le \frac r2$ we have
    \begin{align*}
        f(n,r,s) &\ge n-n \cdot \left(1- \frac s{r-s+1}\right) \cdots  \left(1-\frac 2{r-1}\right)\cdot \left(1-\frac 1r\right)\ge \left(1-e^{-\frac{s(s+1)}{2r}}\right)\cdot n.
    \end{align*}
\end{prop}
\begin{proof}
We will prove the stronger bound by induction on $s$. The basis, for $s=1$ the desired lower bound of $\frac{n}{r}$ follows by looking at a vertex and the colour which appears the most often on the edges incident to it. Suppose now that the result holds for $s-1$ colours, and let $S$ be the largest component one can obtain using $s-1$ colours. 

For any vertex $v$ outside of $S,$ if $v$ sends edges of at most $s-1$ different colours towards $S$, then we can connect $S \cup \{v\}$ using the edges of the star between $v$ and $S$ and their at most $s-1$ colors, contradicting the maximality of $S$. Hence, every vertex outside of $S$ sends edges of at least $s$ different colours to $S$. Now by picking a uniformly random new colour, we will connect any such vertex to $S$ with probability at least $\frac{s}{r-s+1}$. This means that the expected number of vertices we connect to $S$ is $\frac{s}{r-s+1}\cdot (n-|S|)$. Therefore, we can choose a colour in such a way that the new set of $s$ colours connects at least $$|S|+\frac{s}{r-s+1}\cdot (n-|S|)=n-(n-|S|)\cdot \left(1- \frac s{r-s+1}\right)$$ vertices, which after invoking the induction assumption on $|S|$ gives the desired bound.
\end{proof}

\textbf{Remark.} Let us also note that under certain regularity assumptions on the underlying colouring the above argument performs significantly better. For example, if the underlying colouring is approximately uniformly balanced, meaning that no vertex has a degree in any colour larger than $C\cdot \frac{n}{r},$ for some $C\ge 1$, then we obtain a much stronger lower bound. Indeed, if $|S|=D\cdot \frac{n}{r}\le \frac n2$, then by our assumption we know that every vertex needs to send edges of at least $\frac{|S|}{Cn/r}=\frac{D}{C}$ different colours to $S$ so in expectation we reach $\frac{D/C}{r-s+1}\cdot (n-|S|) \ge \frac{1}{2C} \cdot |S|$
new vertices. In particular, the size of the largest component in some $s$ colours is at least $\left(1+\frac{1}{2C}\right)^{s-1} \cdot \frac{n}{r}.$ So, for example, already with $2C\log r$ colours, we can connect $\frac{n}{2}$ vertices, in stark contrast to the general case where we need at least about $\sqrt{r}$ colours to achieve this.

We conclude by showing that the lower bound given in \Cref{prop:f-s-medium} is actually close to tight for the entire range of $s$. 
It is tight up to a constant in the exponent of the lower order term when $s\ge \sqrt{r \log r}$ by our upper bounds for the $g$ function, namely by \Cref{cor:g-s-big-ub}.
For $s \le \sqrt{r}$ the bound in \Cref{prop:f-s-medium} reduces to $f(n,r,s) \ge \Omega\left(\frac{s^2}{r}\right) \cdot n.$ The following proposition shows this is tight up to a logarithmic factor, disproving in quite a strong form a conjecture from \cite{LMP-multicolour}.

\begin{prop}\label{prop:upr-bnd-f-medium-s}
For any $2\le s\le r,$ provided $n \ge \frac rs$ we have
    $$f(n,r,s) \le \frac{8s(s+1)\log r}{r} \cdot n.$$
\end{prop}

\begin{proof}
    We may assume that $8(s+1) \log r\le \frac rs,$ or the inequality is trivial. This implies $\frac{r}{\log r} \ge 8(s+1)s\ge 48$, which in turn implies $\log r \ge 6.$
    Let us define $m:=\floor{\frac{r}{6s}}\ge \frac{r}{7s}$. To start with, we want to show that there exists an $r$-edge colouring of $K_m$ with no $s$ colours connecting more than $ (s+1)\log r$ vertices. To do so let us consider a random colouring in which every edge of our $K_m$ chooses one of the $r$ colours uniformly at random, independently between different edges.

    If there are $k = \ceil{(s+1)\log r}$ vertices connected using up to $s$ colours, then there exists a $k$-vertex tree using up to $s$ colours. There are $\binom{m}{k} \cdot k^{k-2}$ possible choices for the tree, using Cayley's formula, and $\binom{r}{s}$ choices for the set of $s$ colours used on the tree. Given these choices, the probability a given tree is coloured using the given set of $s$ colours is $\left(\frac sr \right)^{k-1}$. Now a union bound implies that the probability that there exists a $k$-vertex tree using up to $s$ colours is at most 
    $$\binom{m}{k} \cdot k^{k-2} \cdot \binom {r}{s} \cdot \left(\frac sr \right)^{k-1} \le  r^{s+1} \cdot \left(\frac{ems}{r}\right)^k \le  2^k\cdot \left(\frac{e}{6}\right)^k < 1,$$
    where in the first inequality we used the standard estimates $\binom{x}{y} \le \min\left\{\left(\frac{ex}{y}\right)^y, x^y\right\}$. This shows that indeed there exists a desired colouring of $K_m$.

    To complete the proof we now take a blow-up of the above colouring of $K_m$ where we replace every vertex with a set of $\floor{\frac nm}$ or $\ceil{\frac nm}$ vertices to obtain an $n$-vertex graph. Let $V(K_m)=\{v_1,\ldots,v_m\}$ and we denote the corresponding sets by $V_1,\ldots,V_m.$
    For $1\le i<j\le m$, we colour the edges between $V_i$ and $V_j$ with the same colour as the edge between $v_i$ and $v_j$. For $1\le i\le m$, we colour the edges inside $V_i$ with an arbitrary colour appearing on an edge incident to $v_i$. This shows that  
    $$f(n,r,s) < (s+1)\log  r \cdot \ceil{\frac nm} \le (s+1)\log r  \cdot \ceil{\frac n{r/(7s)}}\le \frac{8s(s+1)\log r}{r} \cdot n,$$
    completing the proof.
\end{proof}

\section{Concluding remarks and open problems}\label{sec:conc-remarks}
In this paper, we explored a natural ``power of many colours'' generalisation of two classical problems, in the sense first explored by Erd\H{o}s, Hajnal and Rado over 60 years ago in the case of the classical Ramsey questions. While it has predominantly been studied within various instances of the classical Ramsey problem, we believe that extending this approach to encompass other classical colouring problems holds considerable promise and potential for intriguing exploration. 


The first question we considered asks how many vertices can we guarantee to connect using edges of up to $s$ colours in any $r$-edge colouring of $K_n,$ and we denoted the answer to this question by $f(n,r,s)$. The following theorem summarises what is known for this function.

\begin{thm}\label{thm:f-main}
Given integers $1\le s \le {r}$ and $n$ large enough compared to $r,$ we have
{
    \[
  f(n, r, s) =
  \begin{cases}
    \Theta\left(\frac{2^s}{r}\right) \cdot n, & \text{for } s \le (1-o_r(1))\log \log {r}; \\
    \widetilde{\Theta}_r\left(\frac{s^2}{r}\right) \cdot n, & \text{for } s \le \sqrt{r}; \\
    \Theta(n), & \text{for } s= \Theta(\sqrt{r}); \\
    \left(1-e^{-\Theta(s^2/r)}\right)\cdot n, & \text{for } s \ge \sqrt{r \log r}; \\
    \ceil{\left(1-{1}/{\binom{r}{s}}\right)n}, & \text{for } s=\frac{r-1}2 \text{ and } r \text{ odd};\\
    n, & \text{for } s\ge \frac r2.
  \end{cases}
\]
}
\end{thm}

Our new contributions here are the lower bound for the first and the upper bound for the second and fourth part of the above theorem, while the remainder of the theorem follows from the results of Liu, Morris and Prince \cite{LMP-multicolour}. We note that the same picture holds for the more general function $f(n,r,s,k)$ which requires a $k$-connected subgraph in the same setting, and is the answer to the question of Bollob\'as, apart from slightly less precise results in the last two parts, provided $n$ is large enough compared to $r$ and $k$. This holds since our new lower bound in the first part holds in this more general setting as shown in \Cref{thm:f-s-small}, and our new upper bounds also hold since $f(n,r,s,k) \le f(n,r,s)$. Additionally, the matching lower bounds have been established in the more general setting in \cite{LMP-multicolour}.

We note that also for $s=1,2$, tight results are known under certain divisibility assumptions as mentioned in the introduction. Our bound giving the first part above is tight up to an absolute constant factor between $4$ and $8$ depending on divisibility properties of $r.$ With slightly more care, under appropriate divisibility assumptions it can be improved to be tight up to a factor of $2+\eps$ for any $\eps>0$, at a cost of the result only applying for progressively smaller values of $s$ in terms of $r$. This gap of about two has its origins in the fact that our proof uses the case $s=1$ as its starting point and grows by about a factor of two in each step, whereas the aforementioned results show a jump by a factor of $4$ when going to the $s=2$ case. This can possibly be rectified by starting from the $s=2$ case. But due to the far more complicated nature of that proof, the fact that we need a stronger result than the existence of just a single desired subgraph, as well as the fact that we would still be left with a constant factor gap without the divisibility assumptions, we leave this open as a potentially good student project.

Perhaps the most immediate question is to determine whether the $\log r$ term is necessary in the second part of \Cref{thm:f-main} since this is the least precise result we have. Given that our lower bound for the first part of \Cref{thm:f-main} shows that a $(\log r)^{1-o_r(1)}$ factor is necessary when $s \approx \log \log r$ this leads us to the following conjecture
\begin{conj}
    For any $1 \le s \le \frac{\sqrt{r}}{\log r}$ and $n$ large enough we have $f(n,r,s) = \frac{s^2(\log r)^{1-o_r(1)}}r \cdot n$.
\end{conj}
This also highlights the very natural question of how exactly the transition from the answer being $o_r(1) \cdot n$ to $\Theta(1) \cdot n$ happens. See \cite{LMP-multicolour} for a more detailed discussion of this particular question.

Our results show that the two functions $f$ and $g$ behave differently for $s \ll \frac{\sqrt{r}}{\log r}$ and behave similarly when $s \ge \sqrt{r \log r}$. It would be interesting to determine when precisely the functions start to exhibit different behaviour. Let us mention an explicit question in this direction which could be interesting.
\begin{qn}
Given $n$ and $r,$ what is the smallest value of $s$ for which $f(n,r,s)=g(n,r,s)$?
\end{qn}
We know this holds for $s\ge \frac{r-1}{2}$ by the final part of \Cref{thm:f-main}, and it seems already non-trivial to decide it for $s=\frac{r}{2}-1$ (and $r$ even). On the other hand, for $n$ large compared to $r,$ our lower bound of $g(n,r,s)\ge \frac{(1+o_r(1))s}{\sqrt{r}}n$ for $s \ll \sqrt{r}$ compared to the upper bound of $f(n,r,s) \le \frac{(1+o_r(1))s}{\sqrt{2r}} \cdot n$, from \cite{LMP-multicolour}, shows there is a gap of at least about a factor of $\sqrt{2}$ when $s \ll \sqrt{r}$. 

Let us now turn to the question of how many vertices will be touched by edges of some $s$ colours in any $r$-edge colouring of $K_n$, the answer to which we denoted by $g(n,r,s)$. We determine $g(n,r,s)$ up to a constant factor for the whole range, provided $n$ is large enough and even determine it up to lower order terms for the vast majority of the range. With this in mind, the most glaring open problem is to determine how exactly the transition from the answer being $\frac{(1+o_r(1))s}{\sqrt{r}}$ for  $s\ll \sqrt{r}$ to being $(1-o_r(1))n$ for $s\gg \sqrt{r}$ occurs. We do know the behaviour of the $o_r(1)$ term  when $s \ge \sqrt{r \log r}$ but we believe the same behaviour, namely \Cref{cor:g-s-big-ub},  should hold already for  $s \gg \sqrt{r}$. We note however that one cannot hope to improve \Cref{prop:g-big-s-ub} directly since for $s\ll \sqrt{r\log r}$ the hypergraph with $e^{O(s^2/r)}\ll \sqrt{r}\le s$ edges can certainly be covered by $s$ vertices. On the other hand, one might hope to have a hypergraph with more edges but with the property that all vertex subsets of size $r-s$ contain many edges as we sketched in the remark following \Cref{cor:g-s-big-ub}, where we explain how one can obtain our best upper bound of $g(n,r,s) \le \left(1-e^{-O\left(s\sqrt{\frac{\log r}{r}}\right)}\right)n$, for the range $\sqrt{r} \ll s \ll \sqrt{r\log r}$. 

Let us also highlight the following natural problem, extending a classical question of Erd\H{o}s and Lov\'asz \cite{E-L}, from their famous 1975 paper introducing the Local Lemma, which asks for the minimum number of edges in an $s$-uniform intersecting hypergraph with cover number equal to $s$. This question was famously solved by Kahn in 1994 \cite{jeff-paper} and was one of Erd\H{o}s' three favourite combinatorial problems (see e.g.\ \cite{erdos-problems}).

\begin{qn}
    Let $r \ge 2s-1\ge 3$ be integers. What is the minimum number of edges in an $r$-vertex $s$-uniform intersecting hypergraph with cover number equal to $s$?
\end{qn}

The only difference compared to the classical question is that we in addition restrict the number of vertices one may use. So the main new feature of the above question is what happens if we are ``space restricted'' and are not allowed to use as many vertices as we please when building the hypergraph. The $r \ge 2s-1$ condition is here since such a graph trivially cannot exist if $r \le 2s-2$ since any set of $s-1$ vertices would be a cover. On the other hand, for $r=2s-1$ the answer is easily seen to be $\binom{2s-1}{s}$ since the only way to achieve the properties is to take the complete hypergraph. For larger values of $r$, there is always a trivial lower bound of $s$ edges, and the result of Kahn shows that for $r$ large enough in terms of $s$ this is close to tight and the answer is linear in $s$. In fact, the usual double counting argument tells us that we need at least $$\binom{r}{s-1}\Big /\binom{r-s}{s-1}\ge \left(\frac{r}{r-s}\right)^{s-1} = \left(1+\frac{s}{r-s}\right)^{s-1} \ge e^{\frac{s(s-1)}{2(r-s)}}$$ edges, where in the final inequality we used $r \ge 2s$ and $1+x \ge e^{x/2}$ for any real $x\le 1$. This shows that for $s \gg \sqrt{r \log r}$ the answer is growing substantially faster than just linearly in $s$. Our \Cref{prop:f-s-medium} essentially solves a weakening of this question where we allow the uniformity to be somewhat larger than the lower bound on the cover number, which suffices to translate to a rough bound for $g(n,r,s)$. A good understanding of the answer to the above question would translate to a more precise result here as well.

In light of \Cref{lem:upper-bounds} the following natural extremal question arises. 
\begin{qn}
    What is the minimum number of edges in an $r$-vertex intersecting hypergraph if any subset of $m$ vertices needs to contain at least $t$ edges?
\end{qn}
In light of the discussion following \Cref{lem:upper-bounds} if we allow multi-hypergraphs in the above question it would become essentially equivalent to determining $g(n,r,s)$ and in particular most of our results on the $g(n,r,s)$ function came from results we proved for the above question. What turned out to be the most relevant instance for us is if we set $m=r-s$ and if one further restricts the hypergraph to be $(s+1)$-uniform (note that it cannot have an edge of size at most $s$). Indeed, in this case, any edge is contained in $\binom{r-s-1}{s}$ sets of size $r-s$ so if we can achieve that every set of $r-s$ vertices contains the same number of edges then every set of $r-s$ vertices would contain $|E(H)| \binom{r-s-1}{s}/\binom{r}{r-s}$ sets of size $s+1$. Provided we can do so and ensure our hypergraph is intersecting, via \Cref{lem:upper-bounds} we would get an upper bound of $g(n,r,s)\le \left(1-\binom{r-s-1}{s}/\binom{r}{s}\right)n$, which given some divisibility conditions is tight via \Cref{lem:lower-bound-g} with $d=s$. Indeed this is exactly how our precise result at the end of the regime \Cref{g-end-of-regime} proceeds, where we simply use a complete hypergraph. In general, this has a certain design-like flavour and seems to be asking whether a somewhat pseudorandom, intersecting hypergraph exists, a common theme in a number of still open questions from the aforementioned paper of Erd\H{o}s and Lov\'asz \cite{E-L}. In our setting, approximate results in this direction also lead to approximate results on the $g(n,r,s)$ function and are behind most of our more precise results on this function.

\begin{qn}
    For which values of $r,s$ does there exist an $r$-vertex, $(s+1)$-uniform intersecting hypergraph in which all sets of size $r-s$ contain the same positive number of edges?
\end{qn}

Another potentially interesting direction for future work is to consider what happens if one allows base graphs other than the complete one. In particular, one can ask a sparse random analogue, where we replace $K_n$ with the binomial random graph $\G(n,p)$ and let $p$ vary. The famous examples of results of this type are the works of Conlon and Gowers \cite{conlon-gowers} and Schacht \cite{schacht}. Both our questions for $s=1$, the largest monochromatic connected component question (as well as several of its extensions) and the clique covering problem have also been considered in this direction. Various instances have also been considered for expander, pseudorandom and arbitrary large minimum degree base graphs. See e.g.\ \cite{bal-debiasio,discrepancy-of-spanning-trees} for more details.

\textbf{Acknowledgements.} We would like to thank Noah Kravitz for useful discussions during the early phases of this project. We would also like to thank Louis DeBiasio and Henry Liu for their comments. We are also very grateful to the anonymous referee for their very careful reading and for numerous helpful suggestions on the manuscript’s improvement. 
\vspace{-0.2cm}

\providecommand{\MR}[1]{}
\providecommand{\MRhref}[2]{%
\href{http://www.ams.org/mathscinet-getitem?mr=#1}{#2}}


\providecommand{\bysame}{\leavevmode\hbox to3em{\hrulefill}\thinspace}
\providecommand{\MR}{\relax\ifhmode\unskip\space\fi MR }
\providecommand{\MRhref}[2]{%
  \href{http://www.ams.org/mathscinet-getitem?mr=#1}{#2}
}

\end{document}